\documentclass[11pt,leqno]{amsart}
\usepackage{amsmath,amssymb,amsthm}
\usepackage{hyperref}

\newtheorem{theorem}{Theorem}[section]
\newtheorem{lemma}[theorem]{Lemma}

\newtheorem{proposition}[theorem]{Proposition}
\newtheorem{question}{Question}

\theoremstyle{definition}

\newtheorem{example}[theorem]{Example}

\theoremstyle{remark}
\newtheorem{remark}[theorem]{Remark}
\numberwithin{equation}{section}

\def\fnote#1{\footnote}

\def\ignora#1{}
\def\n3#1{\left\vert  \! \left\vert \! \left\vert \, #1 \, \right\vert \!
  \right\vert \! \right\vert }

\newcommand{\pten}{\ensuremath{\widehat{\otimes}_\pi}}
\newcommand{\diam}{\mathop{\mathrm{diam}}\nolimits}
\DeclareMathOperator{\linspan}{span}

\begin{document}

\title{ Daugavet property and separability in Banach spaces }

\author{ Abraham Rueda Zoca }
\address{Universidad de Granada, Facultad de Ciencias.
Departamento de An\'{a}lisis Matem\'{a}tico, 18071-Granada
(Spain) and Instituto de Matem\'aticas de la Universidad de Granada (IEMath-GR)} \thanks{The author was partially supported by Junta de Andaluc\'ia Grants FQM-0199.}\email{ arz0001@correo.ugr.es}
\urladdr{\url{https://arzenglish.wordpress.com}}


\maketitle

\begin{abstract}
We give a characterisation of the separable Banach spaces with the Daugavet property which is applied to study the Daugavet property in the projective tensor product of an $L$-embedded space with another non-zero Banach space. The former characterisation also motivates the introduction of two indices related to the Daugavet property and a short study of them.
\end{abstract}

\section{Introduction}
\bigskip

A Banach space $X$ is said to have the Daugavet property if every rank-one operator $T:X\longrightarrow X$ satisfies the equality
\begin{equation}\label{ecuadauga}
\Vert T+I\Vert=1+\Vert T\Vert,
\end{equation}
where $I$ denotes the identity operator. The previous equality is known as \textit{Daugavet equation} because I. Daugavet proved in \cite{dau} that every compact operator on $\mathcal C([0,1])$ satisfies (\ref{ecuadauga}). Since then, a lot of examples of Banach spaces enjoying the Daugavet property have appeared such as $\mathcal C(K)$ for a compact Hausdorff and perfect topological space $K$, $L_1(\mu)$ and $L_\infty(\mu)$ for a non-atomic measure $\mu$ or the space of Lipschitz functions $Lip(M)$ over a metrically convex space $M$ (see \cite{ikw,kssw,shi,wer} and the references therein for details). Moreover, in \cite{kssw} (respectively \cite{shi}) it appeared a characterisation of the Daugavet property in terms of the geometry of the slices (respectively non-empty weakly open subsets) of $B_X$ (see Theorem \ref{caradaugaclasi} for a formal statement). 

In \cite[Section 6]{wer} D. Werner posed as an open problem how the Daugavet property is preserved by injective or projective tensor products. In \cite[Corollary 4.3]{kkw} an example of a two dimensional complex Banach space $Y$ is given such that $L_\infty^\mathbb C([0,1])\pten Y$ fails the Daugavet property (see \cite[Remark 3.13]{llr2} for real counterexamples failing to fulfill much weaker requirements than the Daugavet property). Concerning to positive results we only know the ones of \cite{br} where, making a strong use of the theory of centralizer and function module representation of Banach spaces, the authors proved that the projective tensor product of a Banach space without minimal $L$-summand with another non-zero Banach space has the Daugavet property. However, up the best of our knowledge, the problem of whether the Daugavet property is preserved by projective tensor products from both factors is still open.

Motivated by this problem and by the recent techniques exposed in \cite[Section 4]{llr2} for the analysis of octahedrality in projective tensor products, in Section \ref{sectifuerte} we will introduce a characterisation of the Daugavet property in separable Banach spaces in terms of coverings of weakly open subsets of the unit ball which will be used to prove the two main results of the paper. On the one hand, we prove in Theorem \ref{tensoLembesepa} that, in presence of the metric approximation property, the Daugavet property is preserved from a separable $L$-embedded Banach space by taking projective tensor products. On the other hand, we prove in Proposition \ref{coropredJBW*triple} that the hypothesis of separability can be eliminated whenever we are dealing with preduals of $JBW^*$-triples with the Daugavet property. In Section \ref{secindithick} we introduce, motivated by Lemma \ref{lemthi} and the thickness index introduced by R. Withley in \cite{whit}, two indices which give a quantitative measurement of how far is a Banach space from having the Daugavet property. We will also study the interrelation of these indices with the Daugavet equation and some stability results concerning to $\ell_p$-sums and to inheritance to subspaces. We finish in Section \ref{secopeque} with some remarks and open questions.
 
\section{Notation and preliminaries}
\bigskip

 We will consider only real Banach spaces. Given a Banach space $X$, we will denote the unit ball and the unit sphere of $X$ by $B_X$ and $S_X$ respectively. Moreover, given $x\in X$ and $r>0$, we will denote $B(x,r)=x+rB_X=\{y\in X:\Vert x-y\Vert\leq r\}$. We will also denote by $X^*$ the topological dual of $X$. Given a bounded subset $C$ of $X$, we will mean by a \textit{slice of $C$} a set of the following form
$$S(C,x^*,\alpha):=\{x\in C:x^*(x)>\sup x^*(C)-\alpha\}$$
where $x^*\in X^*$ and $\alpha>0$. If $X$ is a dual Banach space, the previous set will be a \textit{$w^*$-slice} if $x^*$ belongs to the predual of $X$. Note that finite intersections of slices of $C$ (respectively of $w^*$-slices of $C$) form a basis for the inherited weak (respectively weak-star) topology of $C$.

According to \cite{hww}, a Banach space $X$ is said to be an \textit{$L$-embedded Banach space} if there exists a subspace $Z$ of $X^{**}$ such that $X^{**}=X\oplus_1 Z$. Examples of $L$-embedded Banach spaces are $L_1(\mu)$ spaces, preduals of von Neumann algebras, duals of $M$-embedded spaces or the dual of the disk algebra (see \cite[Example IV.1.1]{hww} for formal definitions and details). 

Given two Banach spaces $X$ and $Y$ we will denote by $L(X,Y)$ the space of all linear and bounded operators from $X$ to $Y$, and we will denote by $X\pten Y$ the projective tensor product of $X$ and $Y$. Moreover, we will say that $X$ has the \textit{metric approximation property} if there exists a net of compact operators $S_\alpha:X\longrightarrow X$ such that $S_\alpha(x)\rightarrow x$ for all $x\in X$. See \cite{rya} for a detailed treatment of the tensor product theory and approximation properties.

The theory of \textit{almost isometric ideals} will be an essential tool for our results related to the Daugavet property in tensor product spaces. Let $Z$ be a subspace of a Banach space $X$.
We say that $Z$ is an \emph{almost isometric ideal} (ai-ideal) in $X$ if
$X$ is locally complemented in $Z$ by almost isometries.
This means that, for each $\varepsilon>0$ and for each
finite-dimensional subspace $E\subseteq X$, there exists a linear
operator $T:E\to Z$ satisfying
\begin{enumerate}
\item\label{item:ai-1}
  $T(e)=e$ for each $e\in E\cap Z$, and
\item\label{item:ai-2}
  $(1-\varepsilon) \Vert e \Vert \leq \Vert T(e)\Vert\leq
  (1+\varepsilon) \Vert e \Vert$
  for each $e\in E$,
\end{enumerate}
i.e. $T$ is a $(1+\varepsilon)$ isometry fixing the elements of $E$.
If the $T$'s satisfy only (\ref{item:ai-1}) and the right-hand side of
(\ref{item:ai-2}) we get the well-known
concept of $Z$ being an \emph{ideal} in $X$ \cite{gks}. Note that the Principle of Local Reflexivity means that $X$ is an ai-ideal in $X^{**}$
for every Banach space $X$. Moreover, the Daugavet property is inherited by ai-ideals
(see \cite{aln2}).  It is also known that, given two Banach spaces $X$ and $Y$ and given an ideal $Z$ in $X$ then $Z\pten Y$ is a closed subspace of $X\pten Y$ (see. e.g. \cite[Theorem 1]{rao}). It is also known that whenever $X^{**}$ or $Y$ has the metric approximation property then $X^{**}\pten Y$ is an isometric subspace of $(X\pten Y)^{**}$ (see \cite[Proposition 2.3]{llr2} and \cite[Theorem 1]{rao}). These two facts will be freely used throughout the Sections \ref{sectifuerte} and \ref{secopeque}. We will also use the following characterisation of ideal in Banach spaces. See \cite[Theorem 1.1]{aln2} and references therein for details.
\begin{theorem}\label{caraidehbope}
Let $X$ be a Banach space and $Y$ be a subspace of $X$. The following assertions are equivalent:
\begin{enumerate}
\item $Y$ is an ideal in $X$.
\item There exists a Hahn-Banach extension operator, that is, an operator $\varphi:Y^*\longrightarrow X^*$ such that, for every $y^*\in Y^*$ and $y\in Y$, it follows that $\Vert \varphi(y^*)\Vert=\Vert y^*\Vert$ and that $\varphi(y^*)(y)=y^*(y)$.
\end{enumerate}
\end{theorem}

Let $X$ be a Banach space. In \cite{whit} R. Whitley defined the following \textit{thickness} index
$$T_W(X):=\inf\left\{r>0:\exists \{x_1,\ldots, x_n\}\subseteq S_X\mbox{ with }S_X\subseteq \bigcup\limits_{i=1}^n B(x_i,r)\right\}.$$
In \cite{cps} it was proved that $T_W(X)$ is equal to 
$$T(X):=\inf\left\{r>0: \exists \{x_1,\ldots, x_n\}\subseteq S_X\mbox{ with } B_X\subseteq \bigcup\limits_{i=1}^n B(x_i,r) \right\}$$
whenever $X$ is infinite-dimensional. Moreover, it is known that $1\leq T(X)\leq 2$ whenever $X$ is an infinite-dimensional Banach space \cite[Lemma 2]{whit}.

Related to the thickness index in Banach spaces is the concept of \textit{octahedral norms}. According to \cite{god}, a Banach space $X$ has an octahedral norm if, for every finite-dimensional subspace $Y\subseteq X$ and every $\varepsilon>0$, there exists $x\in S_X$ such that
$$\Vert y+\lambda x\Vert\geq (1-\varepsilon)(\Vert y\Vert+\vert \lambda\vert)$$
holds for every $\lambda\in\mathbb R$ and every $y\in Y$. It is known \cite{god} that $X$ has an octahedral norm if, and only if, $T(X)=2$. If $X$ is additionally separable, it is known \cite[Lemma 9.1]{gk} that $X$ has an octahedral norm if, and only if, there exists $u\in S_{X^{**}}$ such that
$$\Vert x+u\Vert=1+\Vert x\Vert$$
holds for every $x\in X$.

Finally we shall state the following characterisation of the Daugavet property, proved in \cite{kssw} and \cite{shi}, will be freely used throughout the text.

\begin{theorem}\label{caradaugaclasi}
Let $X$ be a Banach space. The following assertions are equivalent:
\begin{enumerate}
\item $X$ has the Daugavet property.

\item For every $x\in S_X$, every $\varepsilon>0$ and every slice $S$ of $B_X$ there exists $y\in S$ such that $\Vert x+y\Vert>2-\varepsilon$.

\item For every $x\in S_X$, every $\varepsilon>0$ and every non-empty weakly open subset $W$ of $B_X$ there exists $y\in W$ such that $\Vert x+y\Vert>2-\varepsilon$.

\item For every $x^*\in S_{X^*}$, every $\varepsilon>0$ and every $w^*$-slice $S$ of $B_{X^*}$ there exists $y^*\in S$ such that $\Vert x^*+y^*\Vert>2-\varepsilon$.

\item For every $x^*\in S_{X^*}$, every $\varepsilon>0$ and every non-empty weakly-star open subset $W$ of $B_{X^*}$ there exists $y^*\in W$ such that $\Vert x^*+y^*\Vert>2-\varepsilon$.
\end{enumerate}
\end{theorem}

Note that the previous theorem implies that Banach spaces with the Daugavet property have an octahedral norm (see \cite[Lemma 2.8]{kssw} for details).

\section{The Daugavet property in separable Banach spaces and applications}\label{sectifuerte}

\bigskip

As we have pointed out in the previous section, a Banach space $X$ has an octahedral norm if, and only if, whenever there exist $x_1,\ldots, x_n\in X$ such that $B_X\subseteq \bigcup\limits_{i=1}^n B(x_i,r_i)$ then there exists $i\in\{1,\ldots, n\}$ such that $B_X\subseteq B(x_i,r_i)$. We wonder whether a similar statement can be established for the Daugavet property. The following lemma will characterise the above property in terms of a ``thickness kind" condition. In order to see that, we shall introduce a bit of notation. According to \cite{gk}, given a Banach space $X$, it is defined the \textit{ball topology}, and denoted by $b_X$, as the coarsest topology on $X$ so that every closed ball is closed in $b_X$. As a consequence, a basis for the topology $b_X$ is formed by the sets of the following form
$$X\setminus\cup_{i=1}^n B(x_i,r_i),$$
where $x_1,\ldots, x_n$ are elements of $X$ and $r_1,\ldots, r_n$ are positive numbers.

\begin{lemma}\label{lemthi}
Let $X$ be a Banach space. The following assertions are equivalent:
\begin{enumerate}
\item\label{lemthi1} $X$ has the Daugavet property.

\item\label{lemthi2}  Given a non-empty weakly open set $W$ of $B_X$ it follows that, whenever there exist $x_1,\ldots, x_n\in X$ such that $W\subseteq \bigcup\limits_{i=1}^n B(x_i,r_i)$ then there exists $i\in\{1,\ldots, n\}$ such that $r_i\geq 1+\Vert x_i\Vert$. In particular, $W\subseteq B_X\subseteq B(x_i,r_i)$.
\item\label{lemthi3}  For every non-empty $b_X$ open subset $O$ of $B_X$ and for every non-empty weakly open subset $W$ of $B_X$ it follows that $W\cap O\neq \emptyset$.
\end{enumerate}
\end{lemma}

\begin{proof}
(\ref{lemthi1})$\Rightarrow$(\ref{lemthi2}). Pick a non-empty weakly open subset $W$ of $B_X$ and assume that $W\subseteq \bigcup\limits_{i=1}^n B(x_i,r_i)$ for certain $x_1,\ldots, x_n\in X$ and $r_1,\ldots, r_n\in \mathbb R^+$. Let us prove that there exists $i\in\{1,\ldots, n\}$ such that $r_i\geq 1+\Vert x_i\Vert$. Up considering a smaller weakly open set if necessary we can assume that $0\notin W$ and thus $x_1,\ldots, x_n$ are non-zero. Since $X$ has the Daugavet property we conclude, using a similar argument to the one of \cite[Lemma 2.8]{kssw} for weakly open sets, the existence of $y\in W$ such that
$$\Vert x_i-y\Vert>1+\Vert x_i\Vert-\varepsilon$$
holds for every $i\in\{1,\ldots, n\}$. As $y\in W$ then there exists $i\in\{1,\ldots, n\}$ such that $y\in B(x_i,r)$ and thus $r_i\geq 1+\Vert x_i\Vert-\varepsilon$. Since $\varepsilon>0$ was arbitrary it is not difficult to get (\ref{lemthi2}).

(\ref{lemthi2})$\Rightarrow$(\ref{lemthi3}). Consider $O$ to be a non-empty $b_X$ open subset of $B_X$. Up considering a smaller open set, we can assume that $O$ has the following form
$$O:=B_X\setminus\cup_{i=1}^n B(x_i,r_i),$$
for certain $x_1,\ldots, x_n\in X$ and $r_1,\ldots, r_n\in\mathbb R^+$. Consider $W$ to be a non-empty relatively weakly open subset of $B_X$ and assume, by contradiction, that $O\cap W=\emptyset$. Then $W\subseteq \bigcup\limits_{i=1}^n B(x_i,r_i)$. By (\ref{lemthi2}) we get that $B_X\subseteq \bigcup\limits_{i=1}^n B(x_i, r_i)$ and, consequently, $O=\emptyset$, a contradiction. So (\ref{lemthi3}) follows.

(\ref{lemthi3})$\Rightarrow$(\ref{lemthi1}). Pick $x\in S_X$, $\varepsilon>0$ and a slice $S$ of $B_X$. Define $O:=B_X\setminus B(x,2-\varepsilon)$, which is clearly a non-empty $b_X$ open subset of $B_X$. By (\ref{lemthi3}) there exists $y\in S\cap O$, that is, there exists $y\in S$ such that $\Vert y-x\Vert>2-\varepsilon$. Consequently, $X$ has the Daugavet property, so we are done.
\end{proof}

As well as happen with the octahedrality condition, the previous lemma allows us to strengthen the Daugavet property under separability assumptions. 

\begin{theorem}\label{teodaugabidusepa}
Let $X$ be a separable Banach space. The following assertions are equivalent:
\begin{enumerate}
\item\label{teodaugabidusepa1} $X$ has the Daugavet property, that is, for every $x\in S_X$, every non-empty relatively weakly open subset of $B_X$ and every $\varepsilon>0$ there exists $y\in W$ such that $\Vert x+y\Vert>2-\varepsilon$.

\item\label{teodaugabidusepa2} For every non-empty relatively weakly-star open subset $W$ of $B_{X^{**}}$ there exists $u\in S_{X^{**}}\cap W$ such that 
$$\Vert x+u\Vert=1+\Vert x\Vert$$
holds for every $x\in X$.
\end{enumerate}
\end{theorem}

\begin{proof}
(\ref{teodaugabidusepa2})$\Rightarrow$
(\ref{teodaugabidusepa1}). Pick $x\in S_X$, $\varepsilon>0$ and consider a non-empty relatively weakly open subset $W$ of $B_X$. Define $W^*$ to be the weakly-star open subset of $B_{X^{**}}$ defined by $W$ (that is, satisfying that $W^*\cap B_X=W$) and consider $u\in W^*\cap S_{X^{**}}$ as in (\ref{teodaugabidusepa2}). Pick a net $\{x_s\}$ in $B_X$ which is weakly-star convergent to $u$ in $B_{X^{**}}$. On the one hand, because of the weakly-star convergence condition, we can find $s_0$ such that $s\geq s_0$ implies $x_s\in W^*$, and hence $x_s\in W^*\cap B_X=W$. On the other hand, by the weak-star lower semicontinuity of the norm of $X^{**}$, we get
$$2=\Vert x+u\Vert\leq \liminf_s \Vert x_s+x\Vert,$$
so we can find $s\geq s_0$ such that $\Vert x_s+x\Vert>2-\varepsilon$, and (\ref{teodaugabidusepa1}) follows.

 (\ref{teodaugabidusepa1})$\Rightarrow$
 (\ref{teodaugabidusepa2}). Since $X$ is separable then the $b_X$ topology has a countable basis (see e.g. \cite[Introduction]{gk}). Consequently, consider $\{O_n:n\in\mathbb N\}$ to be a basis for the topology $b_X$ of $B_X$. Since $X$ has the Daugavet property then $X$ has an octahedral norm and, consequently, the $b_X$ topology of $B_X$ is not Hausdorff \cite[Theorem 9.1]{gk}. Thus, $\bigcap\limits_{k=1}^n O_k$ is a non-empty $b_X$ open subset of $B_X$ for every $n\in\mathbb N$. Pick $W$ to be a non-empty relatively weakly-star open subset of $B_{X^{**}}$ and $U$ to be another non-empty relatively weakly-star open subset of $B_{X^{**}}$ such that $\overline{U}^{w^*}\subseteq W$. By (\ref{lemthi3}) of Lemma \ref{lemthi} we conclude the existence of $x_n\in (U\cap B_X)\cap \bigcap\limits_{k=1}^n O_k$ for every $n\in\mathbb N$. Since $x_n\in\bigcap\limits_{k=1}^n O_k$ for every $n\in\mathbb N$ we deduce, following word by word the proof of \cite[Lemma 9.l]{gk}, the existence of a $w^*$-cluster point $u$ of the sequence $\{x_n\}$ in $B_{X^{**}}$ such that  
 $$\Vert x-x^{**}\Vert=1+\Vert x\Vert$$
 holds for every $x\in X$. Moreover, since $\{x_n\}$ is contained in $U$ and $u$ is a weak-star cluster point of $\{x_n\}$ we deduce that $u\in \overline{U}^{w^*}\subseteq W$. Consequently, (\ref{teodaugabidusepa2}) follows and the theorem is proved.
\end{proof}

\begin{remark}\label{compateogodnue}
\begin{enumerate}
\item Let $X$ be a separable Banach space. By \cite[Lemma 9.1]{gk} it follows that $X$ has an octahedral norm if, and only if, there exists $u\in S_{X^{**}}$ such that $\Vert x+u\Vert=1+\Vert x\Vert$ holds for every $x\in X$. Theorem \ref{teodaugabidusepa} can be read as follows: $X$ has the Daugavet property if, and only if, the set of such $u\in S_{X^{**}}$ is weak-star dense in $S_{X^{**}}$.
\item Given a separable Banach space $X$ such that $X^*$ additionally has the Daugavet property, Theorem \ref{teodaugabidusepa} can be proved following the argument of \cite[Lemma 2.12]{kssw} for weak-star open subsets instead of $w^*$-slices. 
\end{enumerate}
\end{remark}

As an application we will give some sufficient conditions for a projective tensor product space to enjoy the Daugavet property. For this, we will begin with a characterisation of the Daugavet property in separable $L$-embedded Banach spaces.

\begin{theorem}\label{equiLembe}
Let $X$ be a separable $L$-embedded Banach space. Assume that $X^{**}=X\oplus_1 Z$. Then, the following are equivalent:
\begin{enumerate}
\item \label{equiLembe1} $X^*$ has the Daugavet property.
\item \label{equiLembe2} $X$ has the Daugavet property.
\item \label{equiLembe3} $B_Z$ is weak-star dense in $B_{X^{**}}$.
\end{enumerate}
\end{theorem}

\begin{proof}
(\ref{equiLembe1})$\Rightarrow$
(\ref{equiLembe2}) is obvious.

(\ref{equiLembe2})$\Rightarrow$
(\ref{equiLembe3}). Let $W$ be a non-empty weakly-star open subset of $B_{X^{**}}$ and let us prove that $B_Z\cap W\neq \emptyset$. By Theorem \ref{teodaugabidusepa} we can find $u\in W\cap S_{X^{**}}$ such that
$$\Vert x+u\Vert=1+\Vert x\Vert$$
holds for every $x\in X$. Since $u\in X^{**}$ we can find $x\in X$ and $z\in Z$ such that $u=x+z$. Now
$$1\geq \Vert z\Vert=\Vert -x+(x+z)\Vert=1+\Vert x\Vert.$$
This implies that $x=0$ and, consequently, $u\in B_Z$. So $W\cap B_Z\neq \emptyset$, as desired.

(\ref{equiLembe2})$\Rightarrow$
(\ref{equiLembe3}) follows from \cite[Theorem 2.2]{bm}.
\end{proof}

This result generalises \cite[Theorem 3.2]{bm} under separability assumptions, where the authors proved that a real or complex $JBW^*$-triple $X$ has the Daugavet property if, and only if, its predual $X_*$ (which is an $L$-embedded Banach space) has the Daugavet property.

Now we will apply Theorem \ref{teodaugabidusepa} to study when the projective tensor product of an $L$-embedded Banach space with the Daugavet property enjoys the Daugavet property. For this we shall begin with the following abstract lemma.

\begin{lemma}\label{estadauten1fac}
Let $X$ be a separable Banach space with the Daugavet property and let $Y$ be a non-zero Banach space. Then, for every slice of $B_{X\pten Y}$ $S:=S(B_{X\pten Y},G,\alpha)$ there exists $u\in S_{X^{**}}$ and $y\in S_Y$ such that $(y\circ G)(u)>1-\alpha$ and 
$$\Vert z+u\otimes y\Vert_{(X\oplus \mathbb R u)\pten Y}=1+\Vert z\Vert$$
holds for every $z\in X\pten Y$. Moreover, if $X\oplus \mathbb R u$ is an ideal in $X^{**}$ and either $X^{**}$ or $Y$ has the metric approximation property, then $X\pten Y$ has the Daugavet property.
\end{lemma}

\begin{proof}
Pick $z\in X\pten Y$ and a slice $S:=S(B_{X\pten Y},G,\alpha)$. Consider $x\otimes y\in S\cap S_{X\pten Y}$. Notice that
$$x\otimes y\in S\Leftrightarrow G(x)(y)>1-\alpha\Leftrightarrow x\in S(B_X,y\circ G,\alpha).$$
By Theorem \ref{teodaugabidusepa} there exists $u\in S_{X^{**}}$ such that $u(y\circ G)>1-\alpha$ and that
$$\Vert x+\lambda u\Vert=\Vert x\Vert+\vert\lambda\vert$$
holds for every $x\in X$ and every $\lambda\in\mathbb R$. Denote by $X_u:=X\oplus\mathbb R u$. Now consider $T\in S_{L(X,Y^*)}$ such that $T(z)=\Vert z\Vert$, $y^*\in S_{Y^*}$ such that $y^*(y)=1$ and define $S:X_u\longrightarrow Y^*$ by the equation
$$S(x+\lambda u):=T(x)+\lambda y^*$$
for all $x\in X$ and every $\lambda\in\mathbb R$. Since $X_u$ is isometrically isomorphic to $X\oplus_1\mathbb R$, it is obvious that $\Vert S\Vert\leq 1$. Consequently
$$\Vert z+u\otimes y\Vert_{X_u\pten Y}\geq S(z+u\otimes y)=\Vert z\Vert+y^*(y)=\Vert z\Vert+1.$$
If $X_u$ is an ideal in $X^{**}$ then $X_u\pten Y$ is an isometric subspace of $X^{**}\pten Y$. Moreover, if either $X^{**}$ or $Y$ has the metric approximation property then $X^{**}\pten Y$ is an isometric subspace of $(X\pten Y)^{**}$. Consequently
$$\Vert z+u\otimes y\Vert_{(X\pten Y)^{**}}=1+\Vert z\Vert_{X\pten Y}.$$
Since $u(y\circ G)=(u\otimes y)(G)>1-\alpha$ and $z\in X\pten Y$ was arbitrary we conclude that $X\pten Y$ satisfies (\ref{teodaugabidusepa2}) in Theorem \ref{teodaugabidusepa}. Thus $X\pten Y$ enjoys the Daugavet property, which finishes the proof.
\end{proof}

\begin{remark}\label{remateogentenso} The assumption of Theorem \ref{estadauten1fac} of being $X_u$ an ideal in $X^{**}$ does not hold in general. Indeed, consider a projective tensor product $L_\infty\pten Y$ failing the Daugavet property and that $Y$ has the metric approximation property (see e.g. \cite{llr2} for $Y=\ell_3^3$). Then there exist $z:=\sum_{i=1}^n f_i\otimes y_i\in L_\infty\pten Y$, $\varepsilon_0>0$ and a slice $S:=S(B_{L_\infty\pten Y}, T,\alpha)$ such that, for every $v\in S$, it follows
$$\Vert z+v\Vert\leq \Vert z\Vert+\Vert v\Vert-\varepsilon_0.$$
Now consider $f\otimes y\in S$ and define $E:=\linspan\{f_1,\ldots, f_n,f\}$, which is a finite-dimensional subspace of $L_\infty$. By \cite[Theorem 1.5]{abrahamsen} there exists a separable ai-ideal $W$ in $X$ containing $E$. Since $W$ is an ai-ideal in $X$ then $W$ inherit the Daugavet property \cite[Proposition 3.8]{aln2}. Moreover notice that $T:=S(B_{W\pten Y}, T_{|W},\alpha)$ contains $f\otimes y$. Furthermore, since $W$ is an ai-ideal in $X$ then $\Vert z\Vert_{X\pten Y}=\Vert z\Vert_{W\pten Y}$. Consequently, by the conditions on $z$ and $S$ we deduce that
$$\Vert z+v\Vert_{W\pten Y}\leq \Vert v\Vert+\Vert z\Vert_{W\pten Y}-\varepsilon_0$$
holds for every $v\in T$. This implies that $W$ is a separable Banach space with the Daugavet property and $Y$ is a Banach space with the metric approximation property such that $W\pten Y$ fails the Daugavet property. Then the conclusion follows.\end{remark}

In spite of the previous remark we will exhibit a class of Banach spaces for which Theorem \ref{estadauten1fac} applies.

\begin{theorem}\label{tensoLembesepa}
Let $X$ be a separable $L$-embedded Banach space with the Daugavet property and let $Y$ be a non-zero Banach space. If either $X^{**}$ or $Y$ has the metric approximation property then $X\pten Y$ has the Daugavet property.
\end{theorem}

\begin{proof}
In this case $X^{**}=X\oplus_1 Z$ for some subspace $Z\subseteq X^{**}$ for which $B_Z$ is $w^*$-dense in $B_{X^{**}}$ because of Theorem \ref{equiLembe}. This implies that the element $u$ of the proof of Lemma \ref{estadauten1fac} can be taken in $S_Z$. Pick $u^*\in S_{Z^*}$ such that $u^*(u)=1$. Notice that, if we define $X_u:=X\oplus \mathbb R u$, then $X_u^*=X^*\oplus_\infty \mathbb R u^*$, so the natural inclusion map $\varphi:X_u^*\longrightarrow X^{***}=X^*\oplus_\infty Z^*$ satisfies that
$$\varphi(x^*+\lambda u^*)(x+\lambda u)=(x^*+\lambda u^*)(x+\lambda u)$$
for every $x+\lambda u\in X_u$, which proves that $\varphi$ is a Hahn-Banach extension operator. This implies that $X_u$ is an ideal in $X$ by Theorem \ref{caraidehbope}, so Theorem \ref{estadauten1fac} applies.
\end{proof}

Note that the key to prove Theorem \ref{tensoLembesepa} is that $B_Z$ is $w^*$-dense in $B_{X^{**}}$. Let us exhibit a class of $L$-embedded Banach spaces for which the previous assumption holds, for which we will have to introduce a bit of notation. We recall that a complex
$JB^*$-triple is a complex Banach space $X$ with  a continuous
triple product $\{...\}:X\times X\times X\rightarrow X$ which is
linear and symmetric in  the outer variables, and conjugate-linear
in the middle variable, and satisfies:
\begin{enumerate}
\item For all $x$ in $X$, the mapping $y\rightarrow \{${\it
xxy}$\}$ from $X$ to $X$  is  a  hermitian operator on $X$ and has
nonnegative spectrum. \item The main identity $$
\{ab\{xyz\}\}=\{\{abx\}yz\}-\{x\{bay\}z\}+\{xy\{abz\}\} $$ holds
for all $a,b,x,y,z$  in $X$. \item $\Vert \{xxx\}\Vert =\Vert
x\Vert ^{3}$ for every $x$ in $X$.
\end{enumerate}
Concerning to the condition (1) above, we also recall  that  a  bounded
linear operator $T$ on a complex Banach space $X$ is said to be
hermitian if \linebreak $\Vert \exp (irT)\Vert =1$ for every $r$
in ${\Bbb R}$. Examples of complex $JB^*$-triples are all
$C^*$-algebras under the triple product
$$\{xyz\}:=\frac{1}{2}(xy^*z+zy^*x).$$

Following \cite{ikr}, we define real $JB^*$-triples as norm-closed
real subtriples of complex $JB^*$-triples. Here, by a subtriple we
mean a subspace which is closed under triple products of its
elements. Real $JBW^*$-triples where first introduced as those
real $JB^*$-triples which are dual Banach spaces in such a way
that the triple product becomes separately $w^*$-continuous (see
\cite[Definition 4.1 and Theorem 4.4]{ikr}). Later, it has been
shown in \cite{mp} that the requirement of separate
$w^*$-continuity of the triple product is superabundant. The
bidual of every real (respectively, complex) $JB^*$-triple $X$ is
a $JBW^*$-triple under a suitable triple product which extends the
one of $X$ \cite[Lemma 4.2]{ikr} (respectively,~\cite{din}).

Now we can establish the announced result.

\begin{proposition}\label{coropredJBW*triple}
Let $X$ be a real or complex $JBW^*$-triple, let $X_*$ be its predual and consider a non-zero Banach space $Y$. If $X_*$ has the Daugavet property and either $Y$ or $X^*$ has the metric approximation property, then $X_*\pten Y$ has the Daugavet property.
\end{proposition}
\begin{proof}
In this case $X_*$ is an $L$-embedded Banach space with the Daugavet property. Moreover, it follows that $X^*=X_*\oplus_1 Z$ for some subspace $Z$ of $X^*$. Since $X_*$ has the Daugavet property then $X_*$ does not have any extreme point \cite[Theorem 3.2]{bm}. Consequently, $B_Z$ is $w^*$-dense in $B_{X^*}$, and the proof of Theorem \ref{tensoLembesepa} applies.
\end{proof}

\section{A Daugavet index of thickness}\label{secindithick}

\bigskip

Lemma \ref{lemthi} joint to the definition of the index $T(X)$ motivates the definition of the following index
\begin{equation}\label{indicedauga}
\mathcal T(X):=\inf\left\{r>0:\begin{array}{c}
\exists\ n\in\mathbb N, x_1,\ldots, x_n\in S_X\\
\exists\ \emptyset \neq W\subseteq B_X\mbox{ weakly open } W\subseteq\bigcup\limits_{i=1}^n B(x_i,r)\end{array} \right\}.
\end{equation}
Moreover, in dual Banach spaces, it makes sense considering the following index
\begin{equation}\label{indicedualdauga}
\mathcal T_{w^*}(X):=\inf\left\{r>0:\begin{array}{c}
\exists\ n\in\mathbb N, x_1,\ldots, x_n\in S_X\\
\exists\ \emptyset\neq W\subseteq B_X\mbox{ weak-star open } W\subseteq\bigcup\limits_{i=1}^n B(x_i,r)\end{array} \right\}.
\end{equation}
It is obvious from Lemma \ref{lemthi} that a Banach space $X$ has the Daugavet property if, and only if, $\mathcal T(X)=2$ which in turn is equivalent to the fact that $T_{w^*}(X^*)=2$. It is also clear, from the definition of $\mathcal T(X)$, that $\mathcal T(X)\leq T(X)$, but the inequality may be strict. Indeed, given a non-empty relatively weakly open subset $W$ of $B_X$ such that $W\cap S_X\neq \emptyset$ and $x\in W\cap S_X$, it is clear that $W\subseteq B(x,\diam(W))$. Consequently, the following proposition is clear.
\begin{proposition}\label{indidiamslices}
Let $X$ be a Banahch whose unit ball contains non-empty weakly open subsets of $B_X$ whose diameter is smaller than $\varepsilon$. Then $\mathcal T(X)\leq \varepsilon$. In particular, if $X$ has a dentable unit ball (i.e. the unit ball contains slices of arbitrarily small diameter) then $\mathcal T(X)=0$.
\end{proposition}
Let us now analyse the index $\mathcal T(X)$ for some classical Banach spaces.
\begin{example}\label{ejeextre}
\begin{enumerate}
\item \label{eje1} It is known that $T(\ell_1)=2$ but $\mathcal T(\ell_1)=0$. This show that the inequality $\mathcal T(X)\leq T(X)$ can be strict.
\item \label{eje2} $\mathcal T(c_0)=1$. Indeed, given $r<1$, $\varepsilon:=1-r$ and a non-empty relatively weakly open subset $W\subseteq B_{c_0}$, we will prove that there are no $x_1,\ldots, x_k\in W\cap S_X$ such that $W\subseteq \bigcup\limits_{i=1}^k B(x_i,r)$. For this pick $x_1,\ldots, x_k\in W\cap S_X$ and   $m\in\mathbb N$ such that $n\geq m$ implies $\vert x_i(n)\vert<\varepsilon$. Pick $n\geq m$ such that $x_1+r'e_n\in W$ for $r'>r$ close enough to $r$ (such $n$ exists because $\{e_n\}$ is weakly null). Obviously $x_1+r'e_n\notin \bigcup\limits_{i=1}^k B(x_i,r)$, so $\mathcal T(c_0)\geq 1$. On the other hand, $S(B_{c_0},e_1^*,\alpha)\subseteq B(e_1,1)$, so $\mathcal T(c_0)=1$ as desired. This proves that the converse of Proposition \ref{indidiamslices} does not hold, because it is obvious that every non-empty relatively weakly open subset of $B_{c_0}$ has diameter 2.
\item \label{eje3} $\mathcal T_{w^*}(\ell_\infty)=1$. Indeed, $\mathcal T_{w^*}(\ell_\infty)\leq 1$ as in the previous example. In order to prove the reverse inequality consider a basic non-emtpy weakly-star open subset of $B_{\ell_\infty}$ of the form  $W:=\bigcap\limits_{i=1}^k S(B_{\ell_\infty}, f_i,\alpha_i)$, $0<\varepsilon<\min\limits_{1\leq i\leq k}\alpha_i$ and $x_1,\ldots, x_p\in S_{\ell_\infty}\cap W$. Pick $m\in\mathbb N$ such that $\sum_{n=m+1}^\infty \vert f_i(n)\vert<\varepsilon$ holds for every $i\in\{1,\ldots, k\}$ and pick a finitely-supported sequence $x\in W\cap S_{\ell_\infty}$ such that $f_i(x)>1+\varepsilon-\alpha_i$ holds for every $i\in\{1,\ldots, k\}$. We can assume, with no loss of generality, that $supp(x)\subseteq \{1,\ldots, m\}$. Now define $y\in B_{\ell_\infty}$ by the equation
$$y(p):=\left\{\begin{array}{cc}
x(p) & p\in\{1,\ldots, m\},\\
x_i(q)-sign(x_i(q))e_{q}& q=m+i\in\{m+1,\ldots, m+p\},\\
0 & \mbox{otherwise}.
\end{array} \right.$$
It is clear that $y\in B_{\ell_\infty}$. Moreover, given $i\in\{1,\ldots, k\}$, it follows
$$f_i(y)=\sum_{n=1}^m f_i(n)y(n)+\sum_{n=1}^\infty f_i(n)y(n)\geq$$
$$ f_i(x)-\varepsilon>1+\varepsilon-\alpha_i-
\varepsilon=1-\alpha_i,$$
so $y\in W$. Moreover, given $i\in\{1,\ldots, p\}$, it follows that $\Vert x_i-y\Vert\geq \vert x_i(m+i)-x_i(m+i)+sign(x_i(m+i))\vert=1$. So $\mathcal T_{w^*}(\ell_\infty)\geq 1$, which finishes the proof.
\item\label{ejec(K)} $\mathcal T(\mathcal C([0,1]))=2$ since $\mathcal C([0,1])$ has the Daugavet property. However, the unit ball of $\mathcal C([0,1])^*$ has denting points and, consequently, $\mathcal T(\mathcal C([0,1])^*)=0$. Thus $\mathcal T_{w^*}(\mathcal C([0,1])^{**})<2$.\end{enumerate}
\end{example}
This index still has a relation with the Daugavet equation even when $\mathcal T(X)<2$. The proof of the following result follows the ideas of \cite[Theorem 2.3]{kssw}, but we include the proof for the sake of completeness.
\begin{proposition}\label{indipeqdau}
Let $X$ be a Banach space. Then, for every norm one and weakly compact operator $T:X\longrightarrow X$, it follows
$$\Vert T+I\Vert\geq \mathcal T(X).$$
Similarly, it follows
$$\Vert T+I\Vert\geq \mathcal T_{w^*}(X^*).$$
\end{proposition}

\begin{proof}
Pick $T:X\longrightarrow X$ a weakly compact operator such that $\Vert T\Vert=1$ and $\varepsilon>0$. Then $K=\overline{T(B_X)}$ is weakly compact and, consequently, we can find a denting point $y_0$ of $K$ such that $\Vert y_0\Vert>1-\varepsilon$. For $0<\delta<\varepsilon$ we can find a slice $S:=\{y\in K: y^*(y)>1-\delta\}$ containing $y_0$ and having diameter smaller than $\varepsilon$ (see. e.g. \cite[Theorem 3.6.1]{bou}). For $x^*=T^*(y^*)$, we have $\Vert x^*\Vert=1$ and 
$$T(S(B_X,x^*,\delta))\subseteq S.$$
Now we can find $x\in S(B_X,x^*,\delta)$ such that $\left\Vert x+\frac{y_0}{\Vert y_0\Vert}\right\Vert>\mathcal T(X)-\varepsilon$, so $\Vert x+y_0\Vert>\mathcal T(X)-2\varepsilon$. Moreover $T(x)\in S$ and thus $\Vert T(x)-y_0\Vert<\varepsilon$. Consequently
$$\Vert T+I\Vert\geq \Vert T(x)+x\Vert\geq \Vert x+y_0\Vert-\Vert T(x)-y_0\Vert>\mathcal T(X)-3\varepsilon.$$
Since $\varepsilon>0$ was arbitrary we conclude the desired result.

The second part of the proof follows from the fact that $T^*$ is also weakly compact and then $T^*(B_{X^*})$ has the Radon-Nikodym property, so $T(B_{X^*})$ is $w^*$-dentable (see \cite[Theorem 4.2.13 (f)]{bou}).
\end{proof}

Now we turn to analyse the index $\mathcal T$ with respect to $\ell_p$-sums for $1\leq p\leq \infty$.

\begin{proposition}\label{esta1infi}
Let $X$ and $Y$ be Banach spaces. Then:

\begin{enumerate}
\item\label{thicinfisum} $\mathcal T(X\oplus_\infty Y)\geq \min\{\mathcal T(X),\mathcal T(Y)\}$. Moreover, if $\mathcal T(X\oplus_\infty Y)>1$ then the equality holds.
\item\label{thic1sum} $\mathcal T(X\oplus_1 Y)\leq\min\{\mathcal T(X),\mathcal T(Y)\}$.
\item \label{thicpsum} $\mathcal T(X\oplus_p Y)\leq \left(\frac{(2^\frac{1}{p}+1)^p
1}{2}\right)^\frac{1}{p}$ for every $1\leq p\leq \infty$.
\end{enumerate}
\end{proposition}

\begin{proof}
(\ref{thicinfisum}). Consider a non-empty relatively weakly open subset $W$ of $B_{X\oplus_\infty Y}$, $\varepsilon>0$ and $(x_1,y_1),\ldots, (x_n,y_n)\in S_{X\oplus_\infty Y}$. Define $A:=\{i\in\{1,\ldots, n\}: \Vert x_i\Vert=1\}$ and $B:=\{i\in\{1,\ldots, n\}: \Vert y_i\Vert=1\}$. By \cite[Theorem 4.5]{aln} we can find non-empty weakly open sets $U$ of $B_X$ and $V$ of $B_Y$ such that $U\times V\subseteq W$. Now we can find $x\in U$ such that 
$$\Vert x-x_i\Vert\geq \mathcal T(X)-\varepsilon$$
holds for every $i\in A$. In a similar way we can find $y\in V$ such that
$$\Vert y_i-y\Vert\geq \mathcal T(Y)-\varepsilon$$
holds for every $i\in B$. Clearly $(x,y)\in U\times V\subseteq W$. Moreover, given $i\in\{1,\ldots, n\}=A\cup B$, it follows
$$\Vert (x,y)-(x_i,y_i)\Vert=\max\{\Vert x_i-x\Vert, \Vert y_i-y\Vert\}\geq \min\{\mathcal T(X)-\varepsilon, \mathcal T(Y)-\varepsilon \}.$$
Consequently we get $\mathcal T(X\oplus_\infty Y)\geq \min\{\mathcal T(X), \mathcal T(Y)\}$. For the converse inequality, assume that $\min\{\mathcal T(X),\mathcal T(Y)\}=\mathcal T(X)$ and that $\mathcal T(X\oplus_\infty Y)>1$ and pick $\varepsilon>0$ such that $\mathcal T(X\oplus_\infty Y)-\varepsilon>1$. 
Pick $x_1,\ldots, x_n\in S_X, \varepsilon>0$ and consider a non-empty relatively weakly open subset $W$ of $B_X$. Since $W\times B_Y$ is weakly open in $B_{X\oplus_\infty Y}$, there exists $(x,y)\in W\times B_Y$ such that
$$1<\mathcal T(X\oplus_\infty Y)-\varepsilon\leq \Vert (x_i,0)-(x,y)\Vert=\max\{\Vert x_i-x\Vert,\Vert y\Vert\}$$
holds for every $i\in\{1,\ldots, n\}$. Since $\Vert y\Vert\leq 1$ it follows that $\Vert x_i-x\Vert>\mathcal T(X\oplus_\infty Y)-\varepsilon$ holds for every $i\in\{1,\ldots, n\}$ and, as $x\in W$, we conclude that $\mathcal T(X)\geq \mathcal T(X\oplus_\infty Y)-\varepsilon$. Since $0<\varepsilon<\mathcal T(X\oplus_\infty Y)-1$ was arbitrary we conclude that $\mathcal T(X\oplus_\infty Y)=\min\{\mathcal T(X),\mathcal T(Y)\}$, so (\ref{thicinfisum}) is proved.

(\ref{thic1sum}). Consider $Z:=X\oplus_1 Y$ and assume with no loss of generality that $\min\{\mathcal T(X),\mathcal T(Y)\}=\mathcal T(X)$, and pick $\varepsilon>0$. Then there exists a basic non-emtpy relatively weakly open subset $W=\bigcap\limits_{i=1}^m S(B_X,x_i^*,\alpha)$ of $B_X$ and there are $x_1,\ldots, x_n\in S_X$ such that
$$W\subseteq \bigcup\limits_{i=1}^n B(x_i,\mathcal T(X)+\varepsilon).$$
Now, by the proof of \cite[Proposition 3.1]{abl}, taking $0<\eta<\alpha$ we conclude that $S(B_Z,(x_i^*,0),\eta)\subseteq S(B_X,x_i^*,\alpha)+\eta B_Y$ holds for every $i\in\{1,\ldots, m\}$. Consequently
$$\bigcap\limits_{i=1}^m S(B_Z,(x_i^*,0),\eta)\subseteq \bigcup\limits_{i=1}^n B(x_i,\mathcal T(X)+\varepsilon)\times \eta B_Y\subseteq \bigcup\limits_{i=1}^n B((x_i,0),\mathcal T(X)+\eta+\varepsilon).$$
Since $\alpha$ can be choosen to be arbitrarily small \cite[Lemma 2.1]{ik} we get that $\mathcal T(Z)\leq \mathcal T(X)$.

(\ref{thicpsum}) This follows because $\mathcal T(X\oplus_p Y)\leq T(X\oplus_p Y)\leq \left(\frac{(2^\frac{1}{p}+1)^p
1}{2}\right)^\frac{1}{p}$, where the last inequality was proved in \cite[Proposition 2.7]{hl} \end{proof}

\begin{example}
Let $X:=c_0$, $Y:=\mathbb R$ and $Z:=X\oplus_\infty Y$. Then $Z$ is isometrically isomorphic to $c_0$ and thus $\mathcal T(Z)=1>\min\{\mathcal T(X),\mathcal T(Y)\}=\mathcal T(\mathbb R)=0$. This proves that the inequality in (\ref{thicinfisum}) may be strict if we remove the assumption on $\mathcal T(X\oplus_\infty Y)$.
\end{example}

Let us exhibit now some results related to the index $\mathcal T$ with respect to subspaces.

\begin{proposition}\label{ideacasiso}
Let $X$ be a Banach space and $Y$ be an almost isometric ideal in $X$. Then $\mathcal T(X)\leq \mathcal T(Y)$.
\end{proposition}

\begin{proof}
Pick a positive $\varepsilon>0$, a basic non-empty relatively weakly open subset $W=\bigcap\limits_{j=1}^m S(B_Y,y_j^*,\alpha_j)$ and $y_1,\ldots, y_n\in S_Y$ such that
$$W\subseteq\bigcup\limits_{i=1}^n B(y_i,\mathcal T(Y)+\varepsilon).$$
Consider by \cite[Theorem 1.4]{aln2} a Hahn-Banach extension operator $\varphi:Y^*\longrightarrow X^*$ such that, for all finite dimensional subspaces $E\subseteq X$ and $F\subseteq X^*$ there exists a linear and bounded operator $T:E\longrightarrow Y$ satisfying 
\begin{enumerate}
\item \label{aicondi1} $T(e)=e$ for all $e\in E\cap Y$.
\item \label{aicondi2} $(1+\varepsilon)^{-1}\Vert e\Vert\leq \Vert T(e)\Vert\leq (1+\varepsilon)\Vert T(e)\Vert$ holds for all $e\in E$.
\item \label{aicondi3} $\varphi(y^*)(e)=y^*(T(e))$ for all $e\in E, y^*\in F$.
\end{enumerate}
Let us prove that $U:=\bigcap\limits_{j=1}^m S(B_X,\varphi(y_j^*),\alpha_j)\subseteq \bigcup\limits_{i=1}^n B(y_i,\mathcal (1+\varepsilon)(T(Y)+\varepsilon))$ (notice that, since $W$ is non-empty, so is $U$). To this aim pick $x\in U$ and $\varepsilon>0$, define $E:=\linspan\{y_1,\ldots, y_n,x\}\subseteq X$ and $F:=\linspan\{y_1^*,\ldots, y_m^*\}\subseteq Y^*$ and consider the associated operator $T:E\longrightarrow Y$ satisfying (\ref{aicondi1}), (\ref{aicondi2}) and (\ref{aicondi3}). Now, given $j\in\{1,\ldots, m\}$, we have 
$$1-\alpha_j<\varphi(y_j^*)(x)=y_j^*(T(x)),$$
so $T(x)\in S$. Consequently there exists $i\in\{1,\ldots, n\}$ such that $\Vert T(x)-y_i\Vert\leq \mathcal T(Y)+\varepsilon$. Hence
$$\Vert x-y_i\Vert\leq (1+\varepsilon)\Vert T(x-y_i)\Vert=(1+\varepsilon)\Vert T(x)-y_i\Vert\leq (1+\varepsilon)(\mathcal T(Y)+\varepsilon),$$
which proves the desired inclusion and finishes the proof.
\end{proof}

\begin{remark}\label{ejereciaiideal}
Since every Banach space is an ai-ideal in its bidual, Example \ref{ejeextre} (\ref{ejec(K)}) shows that the inequality in the previous proposition may be strict.
\end{remark}

We will finish the section with another result related to the inheritance to subspaces inspirated in \cite[Theorem 2.2]{blr}.

\begin{proposition}\label{pasofinicodi}
Let $X$ be a Banach space and let $Y$ be a finite-codimensional subspace of $X$. Then $\mathcal T(Y)\geq \mathcal T(X)$.
\end{proposition}

\begin{proof}
Pick a weakly open set $W:=\{y\in Y:\vert y_i^*(y-y_0)\vert<\varepsilon\mbox{ for all }i\in\{1,\ldots, n\}\}$, where $n\in\mathbb N$, $y_1^*,\ldots, y_n^*\in S_{Y^*}$, $y_0\in S_Y$ and $\varepsilon>0$ satisfies that
$$W\cap B_Y\neq \emptyset$$
and pick $y_1,\ldots, y_k\in S_Y$ and $0<\delta<\varepsilon$. Let us find $z\in W\cap B_Y$ such that $\Vert y_i-z\Vert\geq \mathcal T(X)-\delta$ holds for every $i\in\{1,\ldots, k\}$. To this aim assume, up an application of Hahn-Banach theorem, that $y_i^*\in S_{X^*}$ holds for all $i\in\{1,\ldots, n\}$. Define
$$U:=\left\{x\in X:\vert y_i^*(x-y_0)\vert<\varepsilon-\frac{\delta}{4}\mbox{ for all }i\in\{1,\ldots, n\}\right\}.$$
Consider $p:X\longrightarrow X/Y$ to be the quotient map. Now $p(U)$ is a weakly open set of $X/Y$ which contains to $0$. Since $X/Y$ is finite-dimensional there exists a weakly open neighbourhood of $0$ $V$ such that $V\subseteq p(U)$ and that 
$$diam(V)<\frac{\delta}{16}.$$
Consider $B:=p^{-1}(V)\cap U\cap B_X$, which is a non-empty relatively weakly open subset of $B_X$. Since $y_1,\ldots, y_k\in S_Y\subseteq S_X$ we can find $x\in B$ such that
$$\Vert y_i-x\Vert>\mathcal T(X)-\frac{\delta}{16}$$
holds for every $i\in\{1,\ldots, k\}$. As $p(x)\in V$ and $diam(V)<\frac{\delta}{16}$ we can find $u\in Y$ such that $\Vert x-u\Vert<\frac{\delta}{16}$. Define $z:=\frac{u}{\Vert u\Vert}\in S_Y$ and notice that $\Vert x-z\Vert<\frac{\delta}{4}$. Moreover, given $j\in\{1,\ldots, n\}$, we get
$$\vert y_j^*(z-y_0)\vert\leq \vert y_j^*(x-y_0)\vert+\frac{\delta}{4}<\varepsilon,$$
so $z\in W$. Finally, given $i\in\{1,\ldots, k\}$, it follows 
$$\Vert y_i-z\Vert\geq\Vert y_i-x\Vert-\Vert x-z\Vert>\mathcal T(X)-\frac{\delta}{2}>\mathcal T(X)-\delta.$$
Since $\delta>0$ was arbitrary and by definition of the index $\mathcal T(Y)$ we conclude that $\mathcal T(Y)\geq\mathcal T(X)$, so we are done.
\end{proof}
\begin{remark}
The inequality in the previous proposition may be strict. Indeed, consider $Y:=L_1([0,1])$ and $X:=Y\oplus_1\mathbb R$. From Proposition \ref{esta1infi} we get that $\mathcal T(X)\leq \mathcal T(\mathbb R)=0$ while $\mathcal T(Y)=2$.
\end{remark}

\section{Some remarks and open questions}\label{secopeque}

\bigskip

In general, it is false that the property of being an $L$-embedded Banach space is hereditary (see \cite[Chapter IV]{hww}) and, up the best of our knowledge, it is not known whether an ideal in an $L$-embedded Banach space is itself an $L$-embedded Banach space (see \cite[p. 608]{rao}). However, for the class of those $L$-embedded Banach spaces for which every subspace which is an ideal is itself an $L$-embedded Banach space (e.g. von Neumann algebras (see the proof of \cite[Proposition 5]{rao})), the conclusion of Theorem \ref{tensoLembesepa} holds removing the separability assumption.

\begin{proposition}\label{teogenLsumandos}
Let $X$ be an $L$-embedded Banach space with the Daugavet property and let $Y$ be a non-zero Banach space. Assume that every ideal in $X$ is itself an $L$-embedded Banach space. If either $X^{**}$ or $Y$ has the metric approximation property, then $X\pten Y$ has the Daugavet property.
\end{proposition}

\begin{proof} Pick $z:=\sum_{i=1}^n x_i\otimes y_i\in X\pten Y$ and consider a slice $S:=S(B_{X\pten Y},G,\alpha)$. Since $\Vert G\Vert=1$ we can find $x\otimes y\in S\cap S_X$. Define
$$E:=span\{x_1,\ldots, x_n, x\}\subseteq X.$$
Now $E$ is a finite-dimensional subspace of $X$. By \cite[Theorem 1.5]{abrahamsen} we can find an ai-ideal in $X$, say $W$, containing to $E$. Now notice that $\Vert T_{|W}\Vert\geq \Vert T(x)\Vert>1-\alpha$, so we can consider 
$$T:=\{z\in B_{W\pten Y}: G(z)>1-\alpha\},$$
which is a slice of $B_{W\pten Y}$. Moreover, since $W$ is an ai-ideal in $X$, we get that $z\in W\pten Y$ and that $\Vert z\Vert_{W\pten Y}=\Vert z\Vert_{X\pten Y}$. Furthermore, notice that $W^{**}$ has the metric approximation property whenever $X^{**}$ has the metric approximation property because $W^{\circ\circ}$ is $1$-complemented in $X^{**}$. Since $W$ is an $L$-embedded Banach space by the assumptions we conclude from Theorem \ref{tensoLembesepa} that $W\pten Y$ has the Daugavet property and, consequently, there exists $w\in T$ such that $\Vert z+w\Vert_{W\pten Y}>1+\Vert z\Vert-\varepsilon$.
Since $W\pten Y$ is an isometric subspace of $X\pten Y$ we conclude that
$$\Vert z+w\Vert_{X\pten Y}>1+\Vert z\Vert_{X\pten Y}-\varepsilon.$$
Moreover, since $w\in T$ we get that $w\in S$. Hence, $X\pten Y$ enjoys the Daugavet property, as desired.
\end{proof}

In view of the previous Proposition and the fact that, up the best of our knowledge, it is now know whether an ideal in an $L$-embedded Banach space is itself an $L$-embedded Banach space (see \cite[p. 608]{rao}), it is natural to pose the following question.

\begin{question}
Let $X$ be an $L$-embedded space with the Daugavet property and let $Y$ be a non-zero Banach space. If either $X^{**}$ or $Y$ has the metric approximation property, does $X\pten Y$ enjoy the Daugavet property?
\end{question}

It is known that $X\pten Y$ has, at least, an octahedral norm under the assumptions of the previous question \cite[Theorem 4.3]{llr2}.

With respect to Section \ref{secindithick}, in view of the characterisations given in Theorem \ref{caradaugaclasi} and of Proposition \ref{indipeqdau}, it is natural to wonder.
\begin{question}
Let $X$ be a Banach space and $T:X\longrightarrow X$ be a weakly compact operator. Is it true that $\Vert T+I\Vert=\max\{\mathcal T(X),\mathcal T_{w^*}(X^*)\}$?
\end{question}

\bigskip

\textbf{Acknowledgements}: I want to thank Luis Garc\'ia-Lirola for fruitful conversations which improved the paper.

\end{document}